\documentclass[12pt,reqno]{amsart}

\usepackage[letterpaper,margin=1in]{geometry}
\usepackage{lmodern}
\usepackage[T1]{fontenc}
\usepackage{microtype}
\usepackage{amsmath,amssymb,mathtools}
\usepackage{xcolor}

\usepackage{xurl}

\usepackage[
  colorlinks=true,
  linkcolor=blue!70!black,
  citecolor=green!45!black,
  urlcolor=blue!60!black,
  pdfborder={0 0 0},
  pdftitle={Finite-Sample Rigidity for Left-Translated Similarity Orbits and Off-Diagonal Quadratic Realizations},
  pdfauthor={Hiroki Minamide}
]{hyperref}

\newcommand{\doi}[1]{%
  \href{https://doi.org/#1}{doi:\ \nolinkurl{#1}}%
}

\newcommand{\Mat}{\operatorname{Mat}}
\newcommand{\GL}{\operatorname{GL}}

\newcommand{\tr}{\operatorname{tr}}

\newcommand{\Orb}{\operatorname{Orb}}
\newcommand{\Cons}{\operatorname{Cons}}
\newcommand{\Inv}{\operatorname{Inv}}

\newtheorem{theorem}{Theorem}[section]
\newtheorem{proposition}[theorem]{Proposition}
\newtheorem{lemma}[theorem]{Lemma}
\newtheorem{corollary}[theorem]{Corollary}
\theoremstyle{definition}
\newtheorem{definition}[theorem]{Definition}
\theoremstyle{remark}
\newtheorem{remark}[theorem]{Remark}

\makeatletter
\def\@settitle{\begin{center}%
  \baselineskip14\p@\relax
  \bfseries
  \@title
  \end{center}%
}
\def\@setauthors{%
  \begingroup
  \def\thanks{\protect\thanks@warning}%
  \trivlist
  \centering\footnotesize \@topsep30\p@\relax
  \advance\@topsep by -\baselineskip
  \item\relax
  \author@andify\authors
  \def\\{\protect\linebreak}%
  \authors
  \ifx\@empty\contribs
  \else
    ,\penalty-3 \space \@setcontribs
    \@closetoccontribs
  \fi
  \endtrivlist
  \endgroup
}
\makeatother

\title[Finite-Sample Rigidity for Left-Translated Similarity Orbits]
{Finite-Sample Rigidity for Left-Translated Similarity Orbits\\and Off-Diagonal Quadratic Realizations}
\author{Hiroki Minamide}
\address{Department of Liberal Arts, National Institute of Technology, Tokyo College, 1220-2 Kunugida-machi, Hachioji, Tokyo 193-0997, Japan}
\address{Department of Mathematical and Computing Science, School of Computing, Institute of Science Tokyo, 2-12-1 Ookayama, Meguro, Tokyo 152-8550, Japan}
\thanks{\textit{Email address}: \texttt{minamide@tokyo-ct.ac.jp}}
\date{}

\begin{document}

\begin{abstract}
We study finite-sample rigidity for left translates of similarity orbits over commutative rings. Let \(R\) be a commutative ring with \(2\in R^\times\), and let \(C\in\operatorname{Mat}_2(R)\) be cyclic. Given matrices \(G_i=FK_i\), where \(F\in\operatorname{GL}_2(R)\) is common and the underlying points \(K_i\) lie in \(\operatorname{Orb}(C)\), we determine when a finite translated sample has exactly the same multiplier ambiguity as the complete translated orbit. The key structural result is that the orbit differences span \(\mathfrak{sl}_2(R)\), which yields a scalar description of the left stabilizer. We then give a determinantal saturation criterion, expressed directly in the matrices \(G_i\), under which finite-sample consistency reduces to trace equations and one determinant condition. Off-diagonal quadratic operators on pairs of free rank-two modules realize the translated-orbit model exactly. Over fields we obtain sharp sample thresholds under the saturation criterion,
and over finite fields we derive exact saturation probabilities for regular
semisimple and nonzero nilpotent orbits. Torsion restrictions of elliptic isogenies provide a natural arithmetic realization of the rank-two module framework.
\end{abstract}

\maketitle
% Keep the running heads in normal title/name capitalization as well.
\markboth{Hiroki Minamide}{Finite-Sample Rigidity for Left-Translated Similarity Orbits}

% =====================================================================
\section{Introduction}
\label{sec:introduction}
% =====================================================================

Let $R$ be a commutative ring and $C\in\Mat_2(R)$.  Write
\[
  \Orb(C)=\{PCP^{-1}:P\in\GL_2(R)\}
\]
for the similarity orbit of $C$.  The data considered in this paper are
matrices $G_1,\ldots,G_m$ for which there exist a common matrix
$F\in\GL_2(R)$ and underlying orbit points $K_i\in\Orb(C)$ such that
\[
  G_i=FK_i,\qquad i=1,\ldots,m.
\]
The matrices $G_i$ form the given translated sample; neither $F$ nor the
individual points $K_i$ are part of the data.  The $K_i$ may vary
independently, while the same left multiplier $F$ is shared by every member of
the sample.  We ask first how much the complete translated orbit
$F\Orb(C)$ determines $F$, and then when a finite translated sample
already has exactly the same ambiguity.  This is a one-sided problem: the
unknown transformation is a common left multiplier, rather than a conjugating
matrix or a pair of changes of basis.

\subsection{Main results}

Assume throughout the linear-algebraic core that $2\in R^\times$ and that
$C$ is cyclic.  Our basic structural result is
\[
  \operatorname{span}_R
  \{K-K':K,K'\in\Orb(C)\}
  =\mathfrak{sl}_2(R).
\]
Thus the affine directions of a cyclic similarity orbit fill the traceless
submodule.  Combined with the trace pairing, this shows that the left stabilizer of the
complete orbit consists of scalar matrices.

For a finite translated sample $G_i=FK_i$, we identify a saturation condition
under which the sample has exactly the same multiplier ambiguity as the
complete translated orbit.  The key point is that this condition can be
checked directly from the given matrices $G_i$ by a determinantal criterion;
the underlying orbit points $K_i$ need not be known.

Over fields, four affinely saturated samples suffice in the nonzero-trace
case and three linearly saturated samples suffice in the trace-zero cases,
and these bounds are sharp.

\subsection{Realizations and specializations}

The translated-orbit model also arises naturally from quadratic block
operators on pairs of free rank-two modules.  We show that the transverse
graph maps produced by such an operator form exactly a translated similarity
orbit $F\Orb(C)$.  Torsion restrictions of elliptic isogenies provide an
arithmetic realization of this module-theoretic setting.

\subsection{Relation to previous work}

Our problem is adjacent to orbit recovery, matrix equivalence, and
linear-preserver theory, but the transformation class is different.  Orbit
recovery studies an unknown signal modulo a group action
\cite{BandeiraEtAl2023}, while tensor and matrix equivalence seek transformations
between structured objects \cite{GrochowQiao2023}.  Here the similarity orbit
is fixed in advance, the individual orbit points underlying the translated
sample are allowed to vary independently, and one common left multiplier acts
on all of them.

A relevant precursor concerns the linear span of a conjugacy class.  Over an algebraically closed field of characteristic zero,
Skrzy\'nski \cite{Skrzynski2012} proves that the conjugacy class of a
non-scalar matrix spans $\mathfrak{sl}_n$ when the trace is zero and spans
$\Mat_n$ otherwise.  Thus the largeness of a conjugacy class under linear
span is not itself new.  In rank two, the present orbit-difference identity gives an affine counterpart
over an arbitrary commutative ring with $2\in R^\times$: it identifies the full
module of affine directions of a cyclic similarity orbit.  The role of this result in the present paper is to drive a
one-sided left-stabilizer theorem and, subsequently, a finite-subset
determination mechanism.  The common-left-translate problem, its
determinantal saturation criterion, the ring-theoretic scalar ambiguity, and
the sharp finite-field sample results are not treated in \cite{Skrzynski2012}.

Watkins and Li--Pierce study ambient linear maps preserving specified
similarity classes \cite{Watkins1982,LiPierce1994,LiPierce2001}.  Guralnick
classifies invertible linear transformations preserving a finite union of
complete similarity classes over an infinite field \cite{Guralnick1994}.
Here ``finite'' has a different meaning: we keep one prescribed similarity
orbit and ask when a finite subset of one left translate already has the same
multiplier ambiguity as the complete translated orbit.  Hiai treats global
similarity-preserving maps over $\mathbb C$ \cite{Hiai1987}, and Lim extends
this perspective to fields of characteristic zero \cite{Lim1993}.  More
recently, Costara considers linear bijections preserving invertibility on pairs
of similar matrices \cite{Costara2025}, while Schwarz studies ambient
stabilizers of reductive-group orbits \cite{Schwarz2012}.  Our left stabilizer
is the intersection of such an ambient stabilizer with the subgroup of left
multiplications.  These preserver classifications do not address the
finite-subset determination problem considered here.  The determinantal
saturation criterion is specific to this one-sided finite-sample setting.

\subsection{Organization}

Sections~\ref{sec:cyclic-orbits}--\ref{sec:common-multipliers} prove the orbit
rigidity and finite-sample theorem.  Sections~\ref{sec:offdiag}--\ref{sec:transverse-realization}
give the off-diagonal quadratic construction and its exact graph realization.
Sections~\ref{sec:field-specialization}--\ref{sec:finite-geometry} treat fields,
local rings, and finite-field probabilities.  Section~\ref{sec:arithmetic}
connects the module-theoretic model with torsion restrictions of isogenies and
explains the scope of the resulting linear-algebraic tool.  Appendix~A contains
the sharpness constructions, and Appendix~B contains the finite-field counts.

% =====================================================================
\section{Left-translated similarity orbits of cyclic matrices}
\label{sec:cyclic-orbits}
% =====================================================================

Throughout Sections~\ref{sec:cyclic-orbits} and~\ref{sec:common-multipliers},
let $R$ be a commutative ring with identity such that $2\in R^\times$.
We write $\Mat_2(R)$ for the ring of $2\times2$ matrices over $R$ and
$\GL_2(R)$ for its group of units.  Let $I_2$ denote the identity matrix, and let $E_{ij}$ denote the standard
matrix units, with $1$ in the $(i,j)$-entry and $0$ elsewhere.  We write
\[
  \mathfrak{sl}_2(R)
  :=\{X\in\Mat_2(R):\operatorname{tr}(X)=0\}.
\]
Throughout the paper, whenever a matrix $C$ is fixed, we use the standing
notation
\[
  \tau:=\operatorname{tr}(C),\qquad
  \nu:=\det(C),\qquad
  \Delta:=\tau^2-4\nu
\]
when the discriminant is needed.
The purpose of this section is to isolate the rigidity of similarity orbits of
cyclic $2\times2$ matrices.  No field, local-ring, or finiteness hypothesis is
used.

\subsection{Cyclic matrices and companion form}

\begin{definition}
A matrix $C\in\Mat_2(R)$ is \emph{cyclic} if there exists $v\in R^2$ such
that $(v,Cv)$ is an $R$-basis of $R^2$.  Such a vector $v$ is called a
\emph{cyclic vector} for $C$.
\end{definition}

\begin{lemma}[Companion normal form]
\label{lem:cyclic-companion}
If $C\in\Mat_2(R)$ is cyclic, then $C$ is similar over $R$ to
\[
  C_{\tau,\nu}
  :=
  \begin{pmatrix}
    0&-\nu\\
    1&\tau
  \end{pmatrix}.
\]
In particular, two cyclic matrices in $\Mat_2(R)$ with the same characteristic
polynomial are similar.
\end{lemma}

\begin{proof}
Choose a cyclic vector $v$.  By Cayley--Hamilton,
\[
  C^2-\tau C+\nu I_2=0,
\]
so, relative to the basis $(v,Cv)$,
\[
  Cv=0\cdot v+1\cdot Cv,
  \qquad
  C(Cv)=-\nu\,v+\tau\,Cv.
\]
Hence the matrix of $C$ in this basis is $C_{\tau,\nu}$.  The final assertion
follows because two cyclic matrices with the same characteristic polynomial
are both similar to the same companion matrix.
\end{proof}

For $C\in\Mat_2(R)$, let
\[
  \Orb(C)
  :=\{PCP^{-1}:P\in\GL_2(R)\}
\]
be its similarity orbit, and define its \emph{orbit-difference module} by
\[
  D_C
  :=\operatorname{span}_R\{K-K':K,K'\in\Orb(C)\}.
\]
Since $C\in\Orb(C)$, equivalently
\[
  D_C
  =\operatorname{span}_R\{K-C:K\in\Orb(C)\}.
\]

\begin{lemma}[Similarity covariance]
\label{lem:difference-covariance}
For every $P\in\GL_2(R)$,
\[
  D_{PCP^{-1}}=PD_CP^{-1},
  \qquad
  P\mathfrak{sl}_2(R)P^{-1}=\mathfrak{sl}_2(R).
\]
\end{lemma}

\begin{proof}
The first identity follows from
\[
  \Orb(PCP^{-1})=P\Orb(C)P^{-1},
\]
and the second from invariance of trace under similarity.
\end{proof}

\subsection{Generation of the adjoint module by orbit differences}

The following elementary calculation is the key rank-two input.

\begin{lemma}[Companion orbit differences]
\label{lem:companion-orbit-differences}
For the companion matrix $C_{\tau,\nu}$ of
Lemma~\ref{lem:cyclic-companion},
\[
  D_{C_{\tau,\nu}}=\mathfrak{sl}_2(R).
\]
\end{lemma}

\begin{proof}
Define
\[
  U_+:=\begin{pmatrix}1&1\\0&1\end{pmatrix},
  \qquad
  U_-:=\begin{pmatrix}1&-1\\0&1\end{pmatrix}.
\]
A direct computation gives
\[
  U_+C_{\tau,\nu}U_+^{-1}-C_{\tau,\nu}
  =
  \begin{pmatrix}
    1&\tau-1\\
    0&-1
  \end{pmatrix}
\]
and
\[
  U_-C_{\tau,\nu}U_-^{-1}-C_{\tau,\nu}
  =
  \begin{pmatrix}
    -1&-\tau-1\\
    0&1
  \end{pmatrix}.
\]
Their sum is $-2E_{12}$, hence $E_{12}\in D_{C_{\tau,\nu}}$ because
$2\in R^\times$.  Subtracting $(\tau-1)E_{12}$ from the first difference gives
\[
  E_{11}-E_{22}\in D_{C_{\tau,\nu}}.
\]
Moreover,
\[
  \begin{pmatrix}-1&0\\0&1\end{pmatrix}
  C_{\tau,\nu}
  \begin{pmatrix}-1&0\\0&1\end{pmatrix}^{-1}
  -C_{\tau,\nu}
  =2\nu E_{12}-2E_{21}.
\]
Thus $E_{21}\in D_{C_{\tau,\nu}}$.  Since
\[
  E_{12},\qquad E_{21},\qquad E_{11}-E_{22}
\]
form an $R$-basis of $\mathfrak{sl}_2(R)$, we obtain
$\mathfrak{sl}_2(R)\subseteq D_{C_{\tau,\nu}}$.  The reverse inclusion follows
because similarity preserves trace, so every orbit difference is traceless.
\end{proof}

\begin{theorem}[Orbit-difference generation]
\label{thm:orbit-difference-generation}
Let $C\in\Mat_2(R)$ be cyclic.  Then
\[
  D_C=\mathfrak{sl}_2(R).
\]
\end{theorem}

\begin{proof}
By Lemma~\ref{lem:cyclic-companion}, there exists $P\in\GL_2(R)$ such that
\[
  PCP^{-1}=C_{\tau,\nu}.
\]
Lemma~\ref{lem:companion-orbit-differences} gives
$D_{C_{\tau,\nu}}=\mathfrak{sl}_2(R)$, and
Lemma~\ref{lem:difference-covariance} transports this equality back to $C$.
\end{proof}

\begin{corollary}[Affine span of a cyclic similarity orbit]
\label{cor:affine-span-orbit}
Let $C\in\Mat_2(R)$ be cyclic.  The $R$-affine span of $\Orb(C)$ is the
trace hyperplane through $C$:
\[
  \operatorname{Aff}_R(\Orb(C))
  =
  \{X\in\Mat_2(R):\tr(X)=\tr(C)\}.
\]
\end{corollary}

\begin{proof}
By Theorem~\ref{thm:orbit-difference-generation}, the translation
space of the affine span is
\[
  D_C=\mathfrak{sl}_2(R).
\]
Since $C\in\Orb(C)$, the affine span is therefore the trace hyperplane
through $C$, namely
\[
  \{X\in\Mat_2(R):\tr(X)=\tr(C)\}.
\]
\end{proof}

\begin{remark}[Cyclicity]
The cyclicity hypothesis cannot be omitted.  Over any nonzero field, a scalar
matrix $C=aI_2$ is noncyclic, its similarity orbit is the singleton $\{C\}$,
and hence its orbit-difference module is zero.
\end{remark}

\begin{remark}[Adjoint-module interpretation]
\label{rem:adjoint-interpretation}
Because $2\in R^\times$, the trace sequence splits as a sequence of
conjugation modules:
\[
  0\longrightarrow\mathfrak{sl}_2(R)
  \longrightarrow\Mat_2(R)
  \xrightarrow{\operatorname{tr}}R
  \longrightarrow0,
\]
with splitting $a\mapsto (a/2)I_2$.  Hence
\[
  \Mat_2(R)=RI_2\oplus\mathfrak{sl}_2(R),
\]
where the first summand is the trivial module and the second carries the
adjoint action.  Define
\[
  C^\circ:=C-\frac{\operatorname{tr}(C)}2I_2.
\]
Subtracting a scalar matrix preserves cyclicity, so $C^\circ$ is cyclic.  For
every $P\in\GL_2(R)$,
\[
  PC^\circ P^{-1}
  =PCP^{-1}-\frac{\operatorname{tr}(C)}2I_2.
\]
Hence
\[
  \Orb(C^\circ)
  =\left\{K-\frac{\operatorname{tr}(C)}2I_2:K\in\Orb(C)\right\},
  \qquad
  D_{C^\circ}=D_C.
\]
Also $C^\circ$ and $-C^\circ$ have the same characteristic polynomial.
Hence Lemma~\ref{lem:cyclic-companion} gives
$-C^\circ\sim C^\circ$.  Therefore
\[
  2C^\circ=C^\circ-(-C^\circ)\in D_{C^\circ}=D_C,
\]
and $C^\circ\in D_C$ because $2$ is a unit.  It follows that
Theorem~\ref{thm:orbit-difference-generation} is equivalently the assertion
\[
  \operatorname{span}_R
  \{PC^\circ P^{-1}:P\in\GL_2(R)\}
  =\mathfrak{sl}_2(R).
\]
Thus the traceless part of a cyclic matrix generates the adjoint summand under
conjugation.
\end{remark}

\begin{remark}[Characteristic two]
The hypothesis $2\in R^\times$ is not merely cosmetic.  In characteristic
$2$, the scalar line meets $\mathfrak{sl}_2$ nontrivially and the preceding
splitting degenerates.  Orbit-difference generation can then fail; for example,
over $\mathbb F_2$ the cyclic nilpotent matrix
\[
  \begin{pmatrix}0&0\\1&0\end{pmatrix}
\]
has an orbit-difference span of dimension $2$, whereas
$\dim_{\mathbb F_2}\mathfrak{sl}_2(\mathbb F_2)=3$.
\end{remark}

\subsection{The trace pairing and scalar symmetries}

\begin{lemma}[Trace annihilator of the adjoint module]
\label{lem:trace-annihilator}
Under the bilinear pairing
\[
  \langle X,Y\rangle_{\mathrm{tr}}=\operatorname{tr}(XY)
\]
on $\Mat_2(R)$, one has
\[
  \mathfrak{sl}_2(R)^\perp=RI_2.
\]
\end{lemma}

\begin{proof}
Let
\[
  A=\begin{pmatrix}a&b\\c&d\end{pmatrix}
\]
be orthogonal to every traceless matrix.  Pairing with $E_{12}$ and $E_{21}$
gives $c=b=0$, while pairing with
$\operatorname{diag}(1,-1)$ gives $a=d$.  Hence $A$ is scalar.  Conversely,
every scalar matrix is orthogonal to every traceless matrix.
\end{proof}

\begin{definition}
For cyclic $C\in\Mat_2(R)$, define the \emph{scalar-symmetry group}
\[
  \Sigma_C(R)
  :=\{t\in R^\times:tC\text{ is similar to }C\}.
\]
This is a subgroup of $R^\times$: similarity is preserved by multiplication
by a central unit, so $s,t\in\Sigma_C(R)$ imply $st\in\Sigma_C(R)$, and
the same observation applied to an inverse similarity gives
$t^{-1}\in\Sigma_C(R)$.
\end{definition}

\begin{proposition}[Scalar-similarity criterion]
\label{prop:scalar-similarity-general}
Let $C\in\Mat_2(R)$ be cyclic.  Then
\[
  \Sigma_C(R)
  =
  \left\{
    t\in R^\times:
    (t-1)\tau=0,\quad
    (t^2-1)\nu=0
  \right\}.
\]
\end{proposition}

\begin{proof}
If $tC$ is similar to $C$, equality of traces and determinants gives
\[
  t\tau=\tau,
  \qquad
  t^2\nu=\nu,
\]
which yields the displayed conditions.

Conversely, suppose these conditions hold.  Then
\[
  \chi_{tC}(Z)
  =Z^2-t\tau Z+t^2\nu
  =Z^2-\tau Z+\nu
  =\chi_C(Z).
\]
Since $t$ is a unit, $tC$ is cyclic whenever $C$ is cyclic.  By
Lemma~\ref{lem:cyclic-companion}, the cyclic matrices $C$ and $tC$ are similar.
\end{proof}

Define the \emph{left stabilizer} of the similarity orbit by
\[
  \Gamma_C(R)
  :=\{H\in\GL_2(R):H\Orb(C)=\Orb(C)\}.
\]

\begin{theorem}[Cyclic-orbit rigidity]
\label{thm:cyclic-orbit-rigidity}
Let $C\in\Mat_2(R)$ be cyclic.  Then
\[
  \Gamma_C(R)
  =\Sigma_C(R)I_2
  =\left\{
    tI_2:
    t\in R^\times,\ 
    (t-1)\tau=0,\ 
    (t^2-1)\nu=0
  \right\}.
\]
\end{theorem}

\begin{proof}
Let $H\in\Gamma_C(R)$.  For every $K\in\Orb(C)$, one has
$HK\in\Orb(C)$, and therefore
\[
  \operatorname{tr}(HK)=\operatorname{tr}(K)=\tau.
\]
Hence
\[
  \operatorname{tr}((H-I_2)K)=0
\]
for all $K\in\Orb(C)$.  Subtracting the equations corresponding to two
orbit elements shows that $H-I_2$ is orthogonal to the orbit-difference module
$D_C$.  By Theorem~\ref{thm:orbit-difference-generation} and
Lemma~\ref{lem:trace-annihilator},
\[
  H-I_2\in RI_2.
\]
Thus $H=tI_2$ for some $t\in R$.  Since $H$ is invertible, $t\in R^\times$.
The condition $tI_2\in\Gamma_C(R)$ is equivalent to $tC\sim C$, and hence to
$t\in\Sigma_C(R)$.  Proposition~\ref{prop:scalar-similarity-general} gives the
explicit formula.
\end{proof}

% =====================================================================
\section{Common left multipliers and finite-sample rigidity}
\label{sec:common-multipliers}
% =====================================================================

Fix cyclic $C\in\Mat_2(R)$.  The data are $G_1,\ldots,G_m\in\Mat_2(R)$.
We assume that they admit a common-left-translate representation
\[
  G_i=FK_i,
  \qquad K_i\in\Orb(C),
  \qquad F\in\GL_2(R),
  \qquad i=1,\ldots,m.
\]
The matrix $F$ and the individual orbit points $K_i$ are not included in the
data.  The representation records the structural assumption that all members
of the translated sample arise from the same left multiplier, while the
underlying orbit points may vary independently.

\subsection{Complete-orbit ambiguity}

The complete translated orbit is
\[
  F\Orb(C)=\{FK:K\in\Orb(C)\}.
\]

\begin{proposition}[Equality of translated supports]
\label{prop:translated-supports-general}
For $F,F'\in\GL_2(R)$,
\[
  F\Orb(C)=F'\Orb(C)
  \quad\Longleftrightarrow\quad
  F^{-1}F'\in\Gamma_C(R).
\]
Thus the complete translated orbit determines exactly the right coset
\[
  F\Gamma_C(R).
\]
\end{proposition}

\begin{proof}
If the supports are equal, multiplication on the left by $F^{-1}$ gives
\[
  \Orb(C)=F^{-1}F'\Orb(C),
\]
so $F^{-1}F'\in\Gamma_C(R)$.  The converse is immediate from the definition of
$\Gamma_C(R)$.
\end{proof}

\subsection{Determinantal detection of saturation}

For matrices $X_1,\ldots,X_r\in\Mat_2(R)$, choose any $R$-basis of
$\Mat_2(R)$ and let
\[
  M(X_1,\ldots,X_r)\in\Mat_{4\times r}(R)
\]
be the matrix whose $j$th column is the coordinate vector of $X_j$.  Define
\[
  \mathfrak D_3(X_1,\ldots,X_r)\subseteq R
\]
to be the ideal of $R$ generated by the $3\times3$ minors of
$M(X_1,\ldots,X_r)$.  This ideal is independent of the chosen basis.

\begin{lemma}[Left-multiplication invariance]
\label{lem:determinantal-invariance}
For $F\in\GL_2(R)$,
\[
  \mathfrak D_3(FX_1,\ldots,FX_r)
  =\mathfrak D_3(X_1,\ldots,X_r).
\]
\end{lemma}

\begin{proof}
Left multiplication by $F$ is an $R$-module automorphism of $\Mat_2(R)$.
In any chosen basis, it multiplies the coordinate matrix
$M(X_1,\ldots,X_r)$ on the left by an invertible $4\times4$ matrix.
Hence the ideal generated by its $3\times3$ minors is unchanged.
\end{proof}

\begin{lemma}[Determinantal generation criterion]
\label{lem:determinantal-generation}
For $X_1,\ldots,X_r\in\mathfrak{sl}_2(R)$,
\[
  \operatorname{span}_R\{X_1,\ldots,X_r\}=\mathfrak{sl}_2(R)
  \quad\Longleftrightarrow\quad
  \mathfrak D_3(X_1,\ldots,X_r)=R.
\]
\end{lemma}

\begin{proof}
Because $2\in R^\times$,
\[
  E_{11}-E_{22},\qquad E_{12},\qquad E_{21},\qquad I_2
\]
form an $R$-basis of $\Mat_2(R)$, and the first three form an $R$-basis of
$\mathfrak{sl}_2(R)$.  Relative to this ordered basis, every
$X_i\in\mathfrak{sl}_2(R)$ has zero $I_2$-coordinate.  Hence
\[
  M(X_1,\ldots,X_r)
  =\begin{pmatrix}B\\0\end{pmatrix},
  \qquad B\in\Mat_{3\times r}(R).
\]
Thus $\mathfrak D_3(X_1,\ldots,X_r)$ is the ideal generated by the maximal
minors of $B$.  These minors generate the unit ideal if and only if the columns
of $B$ generate $R^3$, equivalently if and only if
$X_1,\ldots,X_r$ generate $\mathfrak{sl}_2(R)$.
\end{proof}

\begin{remark}[Fitting-ideal interpretation]
\label{rem:fitting-saturation}
Let $N=\operatorname{span}_R\{X_1,\ldots,X_r\}\subseteq\mathfrak{sl}_2(R)$.
Then
\[
  \mathfrak D_3(X_1,\ldots,X_r)
  =\operatorname{Fitt}_0\bigl(\mathfrak{sl}_2(R)/N\bigr).
\]
Thus the determinantal criterion is equivalently a zeroth-Fitting-ideal test
for the vanishing of the quotient.  This provides a basis-free
commutative-algebraic interpretation of saturation.
\end{remark}

\begin{definition}
Choose $K_1$ as a reference point in the underlying orbit family
$K_1,\ldots,K_m$.  The family is called \emph{affinely saturated} if
\[
  \operatorname{span}_R\{K_i-K_1:i=2,\ldots,m\}
  =\mathfrak{sl}_2(R).
\]
If $\operatorname{tr}(C)=0$, the family is called \emph{linearly saturated} if
\[
  \operatorname{span}_R\{K_1,\ldots,K_m\}
  =\mathfrak{sl}_2(R).
\]
\end{definition}

The choice of $K_1$ in the affine definition is only a choice of reference
point.  For every $j$,
\[
 \operatorname{span}_R\{K_i-K_j:i\ne j\}
 =\operatorname{span}_R\{K_i-K_1:i\ne1\},
\]
because $K_i-K_j=(K_i-K_1)-(K_j-K_1)$.  Affine saturation is therefore a
property of the orbit family itself.  It is the general condition because all
orbit differences are traceless.  When $\operatorname{tr}(C)=0$, the orbit
itself lies in $\mathfrak{sl}_2(R)$, so linear saturation may require one fewer
point.

\begin{proposition}[Determinantal criterion for saturation]
\label{prop:sample-saturation-general}
Let $G_i=FK_i$ be a translated sample as above.
\begin{enumerate}
  \item The underlying orbit family is affinely saturated if and only if
  \[
    \mathfrak D_3(G_2-G_1,\ldots,G_m-G_1)=R.
  \]
  \item If $\operatorname{tr}(C)=0$, the underlying orbit family is linearly
  saturated if and only if
  \[
    \mathfrak D_3(G_1,\ldots,G_m)=R.
  \]
\end{enumerate}
\end{proposition}

\begin{proof}
Since
\[
  G_i-G_1=F(K_i-K_1),
\]
Part~(1) follows from Lemmas~\ref{lem:determinantal-invariance} and
\ref{lem:determinantal-generation}.  When $\operatorname{tr}(C)=0$, every
$K_i$ is traceless and $G_i=FK_i$, so the same argument proves Part~(2).
\end{proof}

\subsection{Finite-sample rigidity}

For the given matrices $G_1,\ldots,G_m$, define the \emph{orbit-consistency
set}
\[
  \Cons_C(G_1,\ldots,G_m)
  :=\{U\in\GL_2(R):UG_i\in\Orb(C)\text{ for all }i\},
\]
and the \emph{invariant-consistency set}
\[
  \Inv_C(G_1,\ldots,G_m)
  :=\left\{
    U\in\GL_2(R):
    \begin{array}{l}
      \operatorname{tr}(UG_i)=\tau\quad(i=1,\ldots,m),\ 
      \det(UG_1)=\nu
    \end{array}
  \right\}.
\]
The second set is defined by linear trace equations and one determinant
equation.  Since trace and determinant are constant on $\Orb(C)$,
\[
  \Cons_C(G_1,\ldots,G_m)
  \subseteq
  \Inv_C(G_1,\ldots,G_m).
\]

\begin{theorem}[Finite-sample rigidity for common left multipliers]
\label{thm:finite-sample-rigidity-general}
Assume that $C$ is cyclic and that either
\begin{enumerate}
  \item the underlying orbit family is affinely saturated, or
  \item $\tau=0$ and the underlying orbit family is linearly saturated.
\end{enumerate}
Then
\[
\begin{aligned}
  \Cons_C(G_1,\ldots,G_m)
  &=\Inv_C(G_1,\ldots,G_m)\\
  &=\{tF^{-1}:t\in\Sigma_C(R)\}\\
  &=\{tF^{-1}:t\in R^\times,\ (t-1)\tau=0,\ (t^2-1)\nu=0\}.
\end{aligned}
\]
\end{theorem}

\begin{proof}
Let $U\in\Inv_C(G_1,\ldots,G_m)$ and define $A:=UF$.  Since $G_i=FK_i$,
the trace equations become
\[
  \operatorname{tr}(AK_i)=\tau
  \qquad(i=1,\ldots,m).
\]
If the underlying orbit family is affinely saturated, subtracting the equation
for the reference point $K_1$ from the equation for $K_i$ gives
\[
  \operatorname{tr}\bigl(A(K_i-K_1)\bigr)=0.
\]
Affine saturation and Lemma~\ref{lem:trace-annihilator} imply $A=tI_2$ for
some $t\in R$.  If instead $\tau=0$ and the underlying orbit family is
linearly saturated, then $\operatorname{tr}(AK_i)=0$ for all $i$, and the same
lemma again gives $A=tI_2$.

Since $U$ and $F$ are invertible, $t\in R^\times$, and hence $U=tF^{-1}$.
The trace equations give $(t-1)\tau=0$, while the determinant normalization
gives $(t^2-1)\nu=0$.  Proposition~\ref{prop:scalar-similarity-general}
therefore yields $t\in\Sigma_C(R)$.  Hence
\[
  \Inv_C(G_1,\ldots,G_m)
  \subseteq\{tF^{-1}:t\in\Sigma_C(R)\}.
\]

Conversely, let $t\in\Sigma_C(R)$ and define $U:=tF^{-1}$.  Then $UG_i=tK_i$.
Since $tC\sim C$, scalar multiplication by $t$ preserves the whole similarity
orbit: $t\Orb(C)=\Orb(C)$.  Thus $U\in\Cons_C(G_1,\ldots,G_m)$.
Together with $\Cons_C\subseteq\Inv_C$, this proves the result.
\end{proof}

By Proposition~\ref{prop:sample-saturation-general}, both saturation
hypotheses in Theorem~\ref{thm:finite-sample-rigidity-general} are directly
checkable from the given matrices $G_i$.

\begin{corollary}[Finite samples attain the complete-orbit ambiguity]
\label{cor:finite-samples-complete-orbit}
Under the hypotheses of Theorem~\ref{thm:finite-sample-rigidity-general},
\[
  \{U^{-1}:U\in\Cons_C(G_1,\ldots,G_m)\}
  =F\Gamma_C(R).
\]
Consequently, a saturated finite translated sample determines exactly the same
multiplier class as the complete translated orbit $F\Orb(C)$.
\end{corollary}

\begin{proof}
By Theorem~\ref{thm:finite-sample-rigidity-general}, the elements of
$\Cons_C(G_1,\ldots,G_m)$ are $tF^{-1}$ with $t\in\Sigma_C(R)$.
Their inverses are $Ft^{-1}I_2$.  Since
$\Sigma_C(R)$ is a group and $\Gamma_C(R)=\Sigma_C(R)I_2$ by
Theorem~\ref{thm:cyclic-orbit-rigidity}, these inverses form precisely
$F\Gamma_C(R)$.  Proposition~\ref{prop:translated-supports-general} identifies
this coset with the complete-orbit ambiguity.
\end{proof}

% =====================================================================
\section{Off-diagonal quadratic operators}
\label{sec:offdiag}
% =====================================================================

We now construct a module-theoretic source of the translated similarity-orbit
problem.  Let $R$ be a commutative ring, and let $V$ and $W$ be free
$R$-modules of rank two.  The decomposition $V\oplus W$ is part of the fixed
data, and bases of $V$ and $W$ are chosen only when matrix coordinates are
needed.

Fix scalars $\alpha,\gamma,\delta\in R$ and an isomorphism
$F:V\to W$.  Associated with $F$ is the block operator
\[
  T_F=
  \begin{pmatrix}
    \alpha I_V & F^{-1}\\
    \gamma F   & \delta I_W
  \end{pmatrix}
  \in\operatorname{End}_R(V\oplus W).
  \tag{4.1}
\]

\subsection{Normalization and the quadratic algebra}

The form in (4.1) is a normalization of a more general scalar off-diagonal
operator.  Indeed, for $a,c,d\in R$ and $b\in R^\times$,
\[
  \begin{pmatrix}
    aI_V & bF^{-1}\\
    cF   & dI_W
  \end{pmatrix}
\]
takes the form (4.1) after replacing $F$ by $b^{-1}F$ and defining
\[
  \alpha:=a,\qquad \gamma:=bc,\qquad \delta:=d.
\]
Thus (4.1) is a normalization whenever the upper-right scalar coefficient
is a unit.

A direct block multiplication gives
\[
  T_F^2-(\alpha+\delta)T_F
  +(\alpha\delta-\gamma)I_{V\oplus W}=0.
  \tag{4.2}
\]
Accordingly, let
\[
  A_{\alpha,\gamma,\delta}
  :=R[\theta]/
  \bigl(\theta^2-(\alpha+\delta)\theta+\alpha\delta-\gamma\bigr),
  \tag{4.3}
\]
with the residue class of $\theta$ distinguished.  Equation (4.2) gives an
$R$-algebra homomorphism
\[
  \iota_F:A_{\alpha,\gamma,\delta}
  \longrightarrow\operatorname{End}_R(V\oplus W),
  \qquad \theta\longmapsto T_F.
\]
This homomorphism is injective: if
$rI_{V\oplus W}+sT_F=0$, the upper-right block gives $sF^{-1}=0$, hence
$s=0$ and then $r=0$.  Moreover, $\iota_F=\iota_{F'}$ if and only if
$F=F'$, since equality of the distinguished generators gives
$T_F=T_{F'}$ and hence $F^{-1}=F'^{-1}$.  Thus the marked quadratic embedding
retains the off-diagonal isomorphism $F$, even though the ambient endomorphism
algebra is generally noncommutative.

\subsection{The graph companion}

The parameters determine the graph companion
\[
  C_{\alpha,\gamma,\delta}
  :=\begin{pmatrix}
      0&\gamma\\
      1&\delta-\alpha
    \end{pmatrix},
  \qquad
  C:=C_{\alpha,\gamma,\delta}.
  \tag{4.4}
\]
For this $C$,
\[
  \tau=\delta-\alpha,\qquad
  \nu=-\gamma,\qquad
  \Delta=(\delta-\alpha)^2+4\gamma.
\]
Its characteristic polynomial is
\[
  \chi_C(Z)=Z^2-\tau Z+\nu=Z^2-\tau Z-\gamma.
\]
Moreover,
\[
  C\binom{1}{0}=\binom{0}{1},
\]
so $C$ is cyclic over every coefficient ring.  Consequently, whenever
$2\in R^\times$, the rigidity theorems of
Sections~\ref{sec:cyclic-orbits} and~\ref{sec:common-multipliers} apply to
(4.4) with no further regularity assumption.

The defining polynomial in (4.3) has discriminant
\[
  (\alpha+\delta)^2-4(\alpha\delta-\gamma)
  =(\delta-\alpha)^2+4\gamma
  =\Delta.
\]
Thus the marked quadratic algebra and the graph companion carry the same
quadratic discriminant.

% =====================================================================
\section{Transverse modules and exact orbit realization}
\label{sec:transverse-realization}
% =====================================================================

The translated orbit $F\Orb(C)$ arises from two-generated submodules
attached directly to $T_F$.  For $x\in V\oplus W$, define
\[
  L_x:=\langle x,T_Fx\rangle_R\subseteq V\oplus W.
\]
The submodule $L_x$, rather than an ordered pair of generators, is the object
from which the graph map will be extracted.

\subsection{Transverse graph modules}

Let $\pi_V:V\oplus W\to V$ and $\pi_W:V\oplus W\to W$ be the two
projections.

\begin{definition}
A submodule $L\subseteq V\oplus W$ is \emph{$V$-transverse} if
$\pi_V|_L:L\to V$ is an isomorphism.  A vector $x$ is called
\emph{admissible} if $L_x$ is $V$-transverse.
\end{definition}

Every $V$-transverse submodule is the graph of the unique map
\[
  G_L=\pi_W|_L\circ(\pi_V|_L)^{-1}:V\longrightarrow W,
  \qquad
  L=\left\{\binom{v}{G_Lv}:v\in V\right\}.
\]

Because $F$ is an isomorphism, every $x\in V\oplus W$ can be written uniquely
as
\[
  x=\binom{r}{Fz},\qquad r,z\in V.
\]
Writing $T_Fx=(r',s')^{\top}$ gives
\[
  r'=\alpha r+z,
  \qquad
  s'=F(\gamma r+\delta z).
\]
Using fixed bases, define
\[
  P:=(r\;z),\qquad X:=(r\;r'),\qquad Y:=(Fz\;s').
\]
Then
\[
  X=P\begin{pmatrix}1&\alpha\\0&1\end{pmatrix},
  \qquad
  Y=FP\begin{pmatrix}0&\gamma\\1&\delta\end{pmatrix}.
  \tag{5.1}
\]

\begin{proposition}[Admissibility criterion]
\label{prop:admissibility}
For $x$ as above, the following are equivalent:
\begin{enumerate}
  \item $x$ is admissible;
  \item $X\in\GL_2(R)$;
  \item $P\in\GL_2(R)$;
  \item $r,z$ form an $R$-basis of $V$.
\end{enumerate}
\end{proposition}

\begin{proof}
If $X$ is invertible, the first projections of $x$ and $T_Fx$ form a basis of
$V$.  Any relation between the two generators then projects to a trivial
relation, so they form a basis of $L_x$ and $\pi_V|_{L_x}$ is an isomorphism.
Conversely, if $\pi_V|_{L_x}$ is an isomorphism, the first projections generate
$V$, so the square matrix $X$ is surjective and hence invertible.  By (5.1),
$X$ is invertible exactly when $P$ is.  Finally, $P=(r\;z)$ is invertible
exactly when $r,z$ form an $R$-basis of $V$.
\end{proof}

\subsection{Exact realization of a translated similarity orbit}

\begin{theorem}[Exact orbit factorization]
\label{thm:orbit-factorization}
For every admissible $x$,
\[
  G_{L_x}=FPCP^{-1},
\]
where $P=(r\;z)\in\GL_2(R)$ and $C$ is the graph companion (4.4).
\end{theorem}

\begin{proof}
For an admissible $x$, (5.1) gives $G_{L_x}=YX^{-1}$.  Therefore
\[
\begin{aligned}
G_{L_x}
=FP
  \begin{pmatrix}0&\gamma\\1&\delta\end{pmatrix}
  \begin{pmatrix}1&\alpha\\0&1\end{pmatrix}^{-1}
  P^{-1}
=FP
  \begin{pmatrix}0&\gamma\\1&\delta-\alpha\end{pmatrix}
  P^{-1}
 =FPCP^{-1}.
\end{aligned}
\]
\end{proof}

\begin{corollary}[Exact support]
\label{cor:exact-support}
As $x$ varies over the admissible vectors, the set of graph maps is exactly
\[
  F\Orb(C).
\]
Consequently, if $2\in R^\times$, the complete family of transverse graph
modules determines precisely the class $F\Gamma_C(R)$ from
Theorem~\ref{thm:cyclic-orbit-rigidity}.
\end{corollary}

\begin{proof}
Theorem~\ref{thm:orbit-factorization} gives one inclusion.  Conversely, let
$P=(r\;z)\in\GL_2(R)$ and choose $x=(r,Fz)^{\top}$.  Then $x$ is
admissible by Proposition~\ref{prop:admissibility} and has graph map
$FPCP^{-1}$.  The final assertion follows from
Proposition~\ref{prop:translated-supports-general}.
\end{proof}

\begin{remark}[Affine change of the quadratic generator]
\label{rem:affine-generator-change}
For every $\mu\in R$ and $\lambda\in R^\times$,
\[
  \langle x,(\mu I+\lambda T_F)x\rangle_R
  =\langle x,T_Fx\rangle_R=L_x.
\]
Indeed, $(\mu I+\lambda T_F)x$ lies in $L_x$, while
$T_Fx=\lambda^{-1}((\mu I+\lambda T_F)x-\mu x)$.  Thus the transverse module and its
graph map are unchanged by an affine change of the distinguished quadratic
generator.  The orbit realization depends on the resulting two-generated
module, not on this choice of generator.
\end{remark}

The graph map is also independent of the choice of basis of a transverse
module.  If $\binom{X}{Y}$ is the block matrix of any basis of $L$, then
$G_L=YX^{-1}$; replacing the basis by $(XQ,YQ)$ leaves this expression
unchanged.  Hence, once the graph maps satisfy the determinantal criterion of
Proposition~\ref{prop:sample-saturation-general}, the corresponding finite
family reaches the same class $F\Gamma_C(R)$ as the complete translated orbit.

\subsection{Natural sampling over finite fields}

For the quantitative results in Section~\ref{sec:finite-geometry}, the
translated-orbit distribution arises naturally from uniform vectors.

\begin{proposition}[Uniform orbit sampling from uniform vectors]
\label{prop:natural-sampling}
Assume $R=\mathbb F_q$, where $q$ is a prime power.  Sample $x$ uniformly
from $V\oplus W$, write
\[
  x=\binom{r}{Fz},\qquad P:=(r\;z).
\]
Then $P$ is uniform in $\Mat_2(\mathbb F_q)$, $x$ is admissible if and only if
$P\in\GL_2(\mathbb F_q)$, and
\[
  \Pr[x\text{ is admissible}]
  =(1-q^{-1})(1-q^{-2}).
  \tag{5.2}
\]
Conditioned on admissibility, the graph map is uniform on $F\Orb(C)$.
\end{proposition}

\begin{proof}
Since $F$ is bijective, uniform $x$ gives independent uniform $r,z\in V$, so
$P$ is uniform in $\Mat_2(\mathbb F_q)$.  Proposition~\ref{prop:admissibility}
identifies admissibility with invertibility of $P$, and
\[
  |\GL_2(\mathbb F_q)|=(q^2-1)(q^2-q)
\]
gives (5.2).  Conditioned on invertibility, $P$ is uniform in
$\GL_2(\mathbb F_q)$.  The fibers of $P\mapsto PCP^{-1}$ are cosets of the
centralizer of $C$, so the induced orbit point is uniform.
\end{proof}

% =====================================================================
\section{Field specializations and sharp finite-sample rigidity}
\label{sec:field-specialization}
% =====================================================================

Let $k$ be a field of characteristic different from $2$, and let
$C\in\Mat_2(k)$ be cyclic.  The general stabilizer formula becomes completely
explicit over a field.

\begin{corollary}[Field stabilizer classification]
\label{cor:field-stabilizer}
One has
\[
  \Gamma_C(k)=
  \begin{cases}
    \{I_2\},&\tau\neq0,\\[1mm]
    \{\pm I_2\},&\tau=0,\ \nu\neq0,\\[1mm]
    k^\times I_2,&\tau=\nu=0.
  \end{cases}
  \tag{6.1}
\]
In the last case $C$ is necessarily a nonzero nilpotent matrix.
\end{corollary}

\begin{proof}
Apply Theorem~\ref{thm:cyclic-orbit-rigidity}.  If $\tau\neq0$, the relation
$(t-1)\tau=0$ gives $t=1$.  If $\tau=0$ and $\nu\neq0$, the determinant relation
gives $t^2=1$, hence $t=\pm1$.  If $\tau=\nu=0$, both scalar conditions are
vacuous.  A cyclic $2\times2$ matrix with characteristic polynomial $Z^2$ is
nonzero nilpotent.
\end{proof}

The discriminant $\Delta$ determines the ordinary conjugacy type.  If
$\Delta\neq0$, then $C$ is regular semisimple; it is \emph{split} when its
characteristic polynomial splits over $k$, equivalently when $\Delta$ is a
square in $k$, and \emph{nonsplit} otherwise.  If $\Delta=0$, then $C$ is a
nontrivial Jordan matrix.  The intrinsic left ambiguity, however, depends only
on the simpler alternatives in (6.1).

\subsection{Sharp finite-sample rigidity}

The general finite-sample theorem immediately gives the sharp sufficient
sample sizes over fields.  Here ``threshold'' means a uniform determination
bound conditional on the corresponding directly checkable saturation
criterion; arbitrary families of that cardinality need not be saturated.

\begin{corollary}[Field rigidity thresholds]
\label{cor:field-thresholds}
Let $G_i=FK_i$ be a translated sample with underlying orbit points
$K_i\in\Orb(C)$.
\begin{enumerate}
  \item If $\tau\neq0$, four translated matrices whose underlying orbit family
  is affinely saturated determine $F$ uniquely.
  In fact the trace equations
  \[
    \operatorname{tr}(UG_i)=\tau
  \]
  alone have the unique solution $U=F^{-1}$.

  \item If $\tau=0$ and $\nu\neq0$, three translated matrices whose underlying
  orbit family is linearly saturated determine the line $kF^{-1}$ from the
  homogeneous trace equations.  Imposing
  \[
    \det(UG_1)=\nu
  \]
  leaves exactly $U=\pm F^{-1}$.

  \item If $\tau=\nu=0$, three translated matrices whose underlying orbit family
  is linearly saturated determine exactly the line $kF^{-1}$, and therefore
  determine the projective class of $F$.
\end{enumerate}
\end{corollary}

\begin{proof}
Under affine or linear saturation, the proof of
Theorem~\ref{thm:finite-sample-rigidity-general} shows first that every trace
solution has the form $U=tF^{-1}$.  The three conclusions then follow from
$t\tau=\tau$ and, where applicable, $t^2\nu=\nu$.
\end{proof}

Affine saturation requires at least four orbit points, since three independent
differences are needed to generate the three-dimensional space
$\mathfrak{sl}_2(k)$.  Linear saturation requires at least three orbit points.
The next theorem shows that these dimension bounds are genuine uniform bounds,
not artifacts of the proof.

\begin{theorem}[Sharpness of the field thresholds]
\label{thm:field-sharpness}
Over every field of characteristic different from $2$, no uniform rigidity
theorem can replace the sample counts in
Corollary~\ref{cor:field-thresholds} by smaller ones.  More precisely,
there are two-point translated samples compatible with two inequivalent common
left multipliers in each trace-zero regime, and three-point translated samples
compatible with two distinct common left multipliers whenever $\tau\neq0$.
\end{theorem}

The proof is given by explicit constructions in Appendix~\ref{app:sharpness}.

% =====================================================================
\section{Local rings and residue-field detection}
\label{sec:local-specialization}
% =====================================================================

Let $(R,\mathfrak m)$ be a commutative local ring with $2\in R^\times$ and
residue field $k=R/\mathfrak m$, and let $C\in\Mat_2(R)$ be cyclic.
No finiteness or principal-ideal hypothesis is needed in this section.  The
orbit-rigidity and finite-sample theorems have
already been proved over $R$; locality is used only to simplify the sample-based
saturation criterion and the scalar-symmetry group.

\begin{proposition}[Residue-field criterion for saturation]
\label{prop:local-residue-saturation}
Let $X_1,\ldots,X_r\in\mathfrak{sl}_2(R)$ and write bars for reduction modulo
$\mathfrak m$.  Then
\[
  \operatorname{span}_R\{X_1,\ldots,X_r\}=\mathfrak{sl}_2(R)
  \quad\Longleftrightarrow\quad
  \operatorname{span}_k\{\overline X_1,\ldots,\overline X_r\}
  =\mathfrak{sl}_2(k).
\]
Consequently, for translated samples $G_i=FK_i$, affine saturation is detected
by the reduced differences $\overline G_i-\overline G_1$, and in the trace-zero
case linear saturation is detected by the reduced samples $\overline G_i$.
\end{proposition}

\begin{proof}
By Lemma~\ref{lem:determinantal-generation}, generation is equivalent to the
third determinantal ideal being the unit ideal.  Over a local ring this is
equivalent to at least one maximal minor being a unit, which in turn is
equivalent to the corresponding reduced matrix having rank three.  Equivalently,
one may apply Nakayama's lemma to the quotient of $\mathfrak{sl}_2(R)$ by the
generated submodule.  Proposition~\ref{prop:sample-saturation-general} transfers
the criterion from the underlying orbit points $K_i$ to the translated matrices $G_i$.
\end{proof}

The general scalar-symmetry formula remains
\[
  \Sigma_C(R)=
  \{t\in R^\times:(t-1)\tau=0,\ (t^2-1)\nu=0\}.
  \tag{7.1}
\]
When the determinant is a unit, locality collapses the possible involutions.

\begin{corollary}[Unit-determinant local case]
\label{cor:local-unit-det}
Assume $\nu\in R^\times$.  Then
\[
  \Sigma_C(R)=
  \begin{cases}
    \{1\},&\tau\neq0,\\[1mm]
    \{\pm1\},&\tau=0.
  \end{cases}
\]
Hence a saturated finite family determines $F$ exactly when $\tau\neq0$ and up
to sign when $\tau=0$.
\end{corollary}

\begin{proof}
Since $\nu$ is a unit, (7.1) gives $(t-1)(t+1)=0$.  The two factors differ by
$2$, a unit.  In a local ring they cannot both be nonunits, so one factor is a
unit and the other must vanish.  Thus $t=\pm1$.  If $\tau\neq0$, the value
$t=-1$ would give $2\tau=0$, impossible because $2$ is a unit.  If $\tau=0$,
both signs occur.
\end{proof}

This local specialization also clarifies why no principal-ideal hypothesis is
needed: the deterministic theory depends only on free modules, the trace
pairing, and unit detection for maximal minors.

% =====================================================================
\section{Finite-field orbit geometry and saturation probabilities}
\label{sec:finite-geometry}
% =====================================================================

Let $q$ be an odd prime power.  We now quantify how often independently
sampled orbit points over $\mathbb F_q$ satisfy the minimal saturation conditions of
Section~\ref{sec:field-specialization}.  By
Proposition~\ref{prop:natural-sampling}, the same distributions arise from
uniform vectors in the transverse-module model after conditioning on admissibility.

\subsection{Conjugacy types and orbit sizes}

We use the split, nonsplit, and Jordan terminology from
Section~\ref{sec:field-specialization}.

\begin{proposition}[Orbit cardinalities]
\label{prop:orbit-cardinalities}
For a cyclic $C\in\Mat_2(\mathbb F_q)$,
\[
  |\Orb(C)|=
  \begin{cases}
    q(q+1),&C\text{ split regular semisimple},\\
    q(q-1),&C\text{ nonsplit regular semisimple},\\
    q^2-1,&C\text{ nontrivial Jordan}.
  \end{cases}
  \tag{8.1}
\]
\end{proposition}

\begin{proof}
The centralizer in $\GL_2(\mathbb F_q)$ has order $(q-1)^2$ in the split regular
semisimple case, $q^2-1$ in the nonsplit case, and $q(q-1)$ in the nontrivial
Jordan case.  Dividing
\[
  |\GL_2(\mathbb F_q)|=(q^2-1)(q^2-q)
\]
by the corresponding centralizer order gives (8.1).
\end{proof}

The exact saturation probabilities depend on this split/nonsplit geometry.
Their finite-geometric derivation is given in the appendix.

\begin{theorem}[Exact minimal-sample saturation probabilities]
\label{thm:exact-probabilities}
Let $C\in\Mat_2(\mathbb F_q)$ be cyclic, and let the $K_i$ be independent
and uniform on $\Orb(C)$.
\begin{enumerate}
  \item Suppose $\tau\neq0$ and $\Delta\neq0$.  If $C$ is split, then
  \[
    \Pr[K_1,\ldots,K_4\text{ are affinely saturated}]
    =\frac{q-1}{q^2(q+1)^3}\bigl(q^4+3q^3-7q^2+11q-10\bigr).
    \tag{8.2}
  \]
  If $C$ is nonsplit, then
  \[
    \Pr[K_1,\ldots,K_4\text{ are affinely saturated}]
    =\frac{(q-2)(q+1)(q^2-2q-1)}{q^2(q-1)^2}.
    \tag{8.3}
  \]

  \item Suppose $\tau=0$ and $\nu\neq0$.  If $C$ is split, then
  \[
    \Pr[K_1,K_2,K_3\text{ are linearly saturated}]
    =\frac{(q-1)(q^3+2q^2-4q-1)}{q^2(q+1)^2}.
    \tag{8.4}
  \]
  If $C$ is nonsplit, then
  \[
    \Pr[K_1,K_2,K_3\text{ are linearly saturated}]
    =\frac{(q+1)(q^2-3q+1)}{q^2(q-1)}.
    \tag{8.5}
  \]

  \item If $C$ is nonzero nilpotent, then
  \[
    \Pr[K_1,K_2,K_3\text{ are linearly saturated}]
    =\frac{q(q-1)}{(q+1)^2}.
    \tag{8.6}
  \]
\end{enumerate}
\end{theorem}

By Proposition~\ref{prop:sample-saturation-general}, these are also the
success probabilities of the corresponding determinantal tests on the
translated samples $G_i=FK_i$.

In every case in Theorem~\ref{thm:exact-probabilities}, the probability
tends to $1$ as $q\to\infty$.  Thus the sharp deterministic sample size is
also generically sufficient with probability $1-O(q^{-1})$ under uniform orbit
sampling.  The nonzero-eigenvalue Jordan case is not needed for these exact
formulas; its deterministic four-sample threshold is already covered by
Theorem~\ref{thm:field-sharpness}.

\subsection{Transfer to finite local rings}

The same probability calculation controls a finite local ring whenever the
same saturation event is considered after reduction.

\begin{corollary}[Transfer of saturation probabilities]
\label{cor:finite-local-transfer}
Let $R$ be a finite commutative local ring with residue field $\mathbb F_q$,
where $q$ is odd.  Let $C\in\Mat_2(R)$ be cyclic, and write
$\overline C\in\Mat_2(\mathbb F_q)$ for its reduction.  Sample $P_i$ independently and
uniformly from $\GL_2(R)$, and define
\[
  K_i:=P_iCP_i^{-1}.
\]
Then the probability of affine or linear saturation over $R$ equals the
probability of the corresponding spanning event for independent uniform
points of the similarity orbit $\Orb(\overline C)$ over $\mathbb F_q$.
Consequently, formulas~(8.2)--(8.6) apply according to the trace and conjugacy
type of $\overline C$ whenever they describe the same affine or linear
saturation condition over $R$.
\end{corollary}

\begin{proof}
The reduction map $\GL_2(R)\to\GL_2(\mathbb F_q)$ is surjective: any lift of an invertible
residue matrix has determinant outside the maximal ideal and is therefore
invertible.  Its fibers are cosets of the kernel, so the reductions
$\overline P_i$ are independent and uniform in $\GL_2(\mathbb F_q)$.  Hence
\[
  \overline K_i=\overline P_i\,\overline C\,\overline P_i^{-1}
\]
are independent and uniform on $\Orb(\overline C)$.
Proposition~\ref{prop:local-residue-saturation} identifies the saturation event
over $R$ with the corresponding spanning event over $\mathbb F_q$.
\end{proof}

\begin{remark}
If $\tau$ is nonzero in $R$ but reduces to zero in $\mathbb F_q$, affine saturation
remains the relevant condition over $R$.  Thus transfer of a saturation
probability should not be confused with transfer of the minimal sample count
for a different trace type over the residue field.
\end{remark}

% =====================================================================
\section{Arithmetic realization and scope of the method}
\label{sec:arithmetic}
% =====================================================================

The preceding theory is entirely linear-algebraic.  Its input is a pair of
free rank-two modules over a commutative ring $R$, together with an
isomorphism $F$ between them.  We first record an arithmetic source of exactly
this module-theoretic data.  No statement about abelian varieties or isogenies
is used in the proofs of Sections~\ref{sec:cyclic-orbits}--\ref{sec:finite-geometry}.

\subsection{Torsion-module realization}

\begin{proposition}[Elliptic torsion realization]
\label{prop:elliptic-torsion-realization}
Let $k$ be a field, fix an algebraic closure $\overline k$ of $k$, and let
$E_1,E_2$ be elliptic curves over $k$.  Let $m\ge2$ be prime to
$\operatorname{char}(k)$, and let $f:E_1\to E_2$ be an isogeny of degree $n$
with $\gcd(m,n)=1$.  Define $R:=\mathbb Z/m\mathbb Z$.  Then
\[
  F:=f|_{E_1[m](\overline k)}:
  E_1[m](\overline k)\xrightarrow{\sim}E_2[m](\overline k)
\]
is an isomorphism of free rank-two $R$-modules.  If
$n\bar n\equiv1\pmod m$, then
\[
  F^{-1}=[\bar n] \,\widehat f|_{E_2[m](\overline k)}.
  \tag{9.1}
\]
\end{proposition}

\begin{proof}
Since $m$ is prime to $\operatorname{char}(k)$, the geometric torsion modules
satisfy
\[
  E_i[m](\overline k)\simeq(\mathbb Z/m\mathbb Z)^2.
\]
The dual-isogeny identities
$\widehat f f=[n]_{E_1}$ and $f\widehat f=[n]_{E_2}$ show that multiplication
by $n$ is the composite in either direction.  Since $n$ is a unit modulo $m$,
restriction to geometric $m$-torsion gives (9.1), and in particular shows that
$F$ is an isomorphism; see \cite{Milne1986}.
\end{proof}

Thus the identification with the notation of Sections~\ref{sec:offdiag} and
\ref{sec:transverse-realization} is simply
\[
  V=E_1[m](\overline k),\qquad
  W=E_2[m](\overline k),\qquad
  R=\mathbb Z/m\mathbb Z,
\]
with $F$ given by the torsion restriction of $f$.  When $m$ is odd,
$2\in R^\times$, so all rigidity results of Sections~\ref{sec:cyclic-orbits}
and~\ref{sec:common-multipliers} apply to this realization.

The entire block operator of Section~\ref{sec:offdiag}, not only the map $F$,
can be realized on torsion.  Choose integer lifts
$\widetilde\alpha,\widetilde\gamma,\widetilde\delta\in\mathbb Z$ of
$\alpha,\gamma,\delta\in R$, and choose $u\in\mathbb Z$ whose residue
class modulo $m$ is $\bar n$.  The product endomorphism
\[
  \widetilde T_f=
  \begin{pmatrix}
    [\widetilde\alpha]_{E_1} & [u]\widehat f\\
    [\widetilde\gamma]f & [\widetilde\delta]_{E_2}
  \end{pmatrix}
  \in \operatorname{End}_{\overline k}(E_1\times E_2)
\]
restricts on
$E_1[m](\overline k)\oplus E_2[m](\overline k)$ to
\[
  \begin{pmatrix}
    \alpha I & F^{-1}\\
    \gamma F & \delta I
  \end{pmatrix}=T_F.
\]
Thus the off-diagonal quadratic operator and the transverse-module
construction of Sections~\ref{sec:offdiag}--\ref{sec:transverse-realization}
are themselves obtained by torsion restriction from a product endomorphism.

\subsection{Off-diagonal maps in product endomorphism algebras}

The preceding proposition explains how an isogeny produces the module map
$F$.  Such maps occur naturally as off-diagonal components of product
endomorphisms.  For abelian varieties $A_1,A_2$ over $k$,
\[
  \operatorname{End}^0_k(A_1\times A_2)
  \cong
  \begin{pmatrix}
    \operatorname{End}^0_k(A_1)&\operatorname{Hom}^0_k(A_2,A_1)\\
    \operatorname{Hom}^0_k(A_1,A_2)&\operatorname{End}^0_k(A_2)
  \end{pmatrix},
  \tag{9.2}
\]
where $\operatorname{End}^0=\operatorname{End}\otimes_{\mathbb Z}\mathbb Q$
and $\operatorname{Hom}^0=\operatorname{Hom}\otimes_{\mathbb Z}\mathbb Q$.
The diagonal algebras in (9.2) need not be commutative, and the off-diagonal
$\operatorname{Hom}^0$ spaces are bimodules rather than coefficient rings.
The commutative ring used by the rigidity theory is instead the torsion
coefficient ring $R=\mathbb Z/m\mathbb Z$.  Formula (9.2) serves as a
structural source of the off-diagonal morphisms whose torsion restrictions
produce the maps $F$ and $F^{-1}$ above; it is not being substituted for $R$.

As a motivating example, not used elsewhere in the paper, let $J$ be a
superspecial genus-two Jacobian over $k$.  After base change to $\overline k$,
choose an unpolarized isomorphism
\[
  J_{\overline k}\simeq E_1\times E_2
\]
with supersingular elliptic curves $E_1,E_2$
\cite{AchterPries2015}.  The resulting block decomposition of
$\operatorname{End}^0_{\overline k}(J_{\overline k})$ contains rational
$\operatorname{Hom}^0$ spaces off the diagonal.  Any isogeny between the
elliptic factors whose degree is prime to $m$ therefore supplies, after
restriction to $m$-torsion, the finite-module data of
Proposition~\ref{prop:elliptic-torsion-realization}.

\subsection{Scope and extensions}

The rigidity results of Sections~\ref{sec:cyclic-orbits}--
\ref{sec:transverse-realization} turn an off-diagonal module map into a
finite-sample determination problem with an explicit determinantal criterion.
Once saturation holds, the finite family has exactly the same multiplier
ambiguity as the complete translated orbit.

The theory isolates this rank-two linear-algebraic step once the module data
are available; it does not reconstruct a global isogeny or an integral
endomorphism order.  Natural extensions include higher rank, characteristic
$2$, and higher-degree marked algebras.

% =====================================================================
\appendix
\renewcommand{\thetheorem}{\Alph{section}.\arabic{theorem}}
\section{Sharpness constructions}
\label{app:sharpness}
% =====================================================================

The following explicit constructions establish the lower bounds asserted in
Theorem~\ref{thm:field-sharpness}.

\begin{proof}[Proof of Theorem~\ref{thm:field-sharpness}]
By Lemma~\ref{lem:cyclic-companion}, similarity allows us to work with
\[
  C:=\begin{pmatrix}0&\gamma\\1&\tau\end{pmatrix},
  \qquad \nu=-\gamma.
\]
For $s\in k$, define
\[
  U_s:=
  \begin{pmatrix}
    1&s\\
    0&1
  \end{pmatrix}.
\]
\medskip
\noindent\emph{Trace-zero semisimple case.}
Suppose first that $\tau=0$ and $\nu\neq0$, so $\gamma\neq0$.  Choose
$s\in k^\times$ and define
\[
  K_1:=C,\qquad K_2:=U_sCU_s^{-1}, \qquad 
  A:=
  \begin{pmatrix}
    1&-s\\
    s/\gamma&1-s^2/\gamma
  \end{pmatrix}.
\]
A direct calculation gives
\[
  \det A=1,\qquad
  \operatorname{tr}(AK_1)=\operatorname{tr}(AK_2)=0,
\]
and
\[
  \det(AK_1)=\det(AK_2)=\det(C)=-\gamma.
\]
The common characteristic polynomial is $Z^2-\gamma$, whose discriminant
$4\gamma$ is nonzero.  Hence both products are cyclic and, by
Lemma~\ref{lem:cyclic-companion}, belong to $\Orb(C)$.  Since $A$ is
non-scalar, the same two-point translated sample is compatible with $F_0=I_2$ and
$F_1=A^{-1}$, which are not related by the sign stabilizer.

\medskip
\noindent\emph{Nilpotent case.}
Choose $s\in k^\times$ and take
\[
  C:=K_1:=\begin{pmatrix}0&0\\1&0\end{pmatrix},
  \qquad
  K_2:=U_sCU_s^{-1}
  =\begin{pmatrix}s&-s^2\\1&-s\end{pmatrix}.
\]
Choose $c\in k^\times$ with $cs\neq1$; such a choice exists because $k$ has
odd characteristic and only the value $c=s^{-1}$ is excluded.  Define
\[
  A:=\begin{pmatrix}1&0\\c&1-cs\end{pmatrix}.
\]
Then
\[
  \det A=1-cs\neq0,
\]
and
\[
  AK_1=
  \begin{pmatrix}0&0\\1-cs&0\end{pmatrix},
  \qquad
  AK_2=
  \begin{pmatrix}s&-s^2\\1&-s\end{pmatrix}.
\]
Both matrices are nonzero with trace and determinant zero.  By
Cayley--Hamilton they are nonzero nilpotent matrices, hence lie in the unique
nonzero nilpotent similarity orbit.  Since $A$ is non-scalar, two translated matrices do
not determine the projective class.

\medskip
\noindent\emph{Nonzero-trace regular semisimple case.}
Assume $\tau\neq0$ and $\Delta\neq0$.  By
Lemma~\ref{lem:cyclic-companion}, we may again take
\[
  C:=\begin{pmatrix}0&\gamma\\1&\tau\end{pmatrix},
  \qquad \nu=-\gamma.
\]
Choose $s\in k^\times$ with $s\neq\tau$ and define
\[
  K_1:=C,\qquad
  K_2:=U_sCU_s^{-1},\qquad
  K_3:=
  \begin{pmatrix}-1&0\\0&1\end{pmatrix}
  C
  \begin{pmatrix}-1&0\\0&1\end{pmatrix}^{-1}.
\]
Using the coordinates $(1,1)$, $(2,1)$, and $(2,2)$ gives a
$3\times3$ minor equal to
\[
  -2s\tau,
\]
so these three orbit elements are linearly independent.  Define
\[
  H_s:=\begin{pmatrix}s-\tau&-\gamma\\1&0\end{pmatrix}.
\]
One checks directly that
\[
  \operatorname{tr}(H_sK_i)=0
  \qquad(i=1,2,3),
\]
so the common trace equations contain the affine line
\[
  U(t):=I_2+tH_s.
\]
By nondegeneracy of the trace pairing, these three trace equations are
independent on the four-dimensional space $\Mat_2(k)$, so their common
solution set is exactly this affine line.  Its determinant is
\[
  \det U(t)=1+t(s-\tau)+\gamma t^2.
\]

If $\gamma=0$, then $\det C=0$ and
\[
  \det U(t)=1+t(s-\tau).
\]
At most one nonzero value of $t$ makes $U(t)$ singular, so there exists
$t\in k^\times$ with $U(t)$ invertible.  For such $t$, each $U(t)K_i$ has
trace $\tau$ and determinant $0$.  Its characteristic polynomial is
$Z(Z-\tau)$, which has distinct roots because $\tau\neq0$; hence
$U(t)K_i\in\Orb(C)$.

If $\gamma\neq0$, the nonzero value
\[
  t_*:=\frac{\tau-s}{\gamma}
\]
satisfies $\det U(t_*)=1$.  Thus every $U(t_*)K_i$ has the same trace and
determinant as $C$.  Since $\Delta\neq0$, the common characteristic polynomial
is separable, so all three products lie in $\Orb(C)$.  In either case the
same translated sample is compatible with $I_2$ and a distinct left multiplier.

\medskip
\noindent\emph{Nonzero-eigenvalue Jordan case.}
Assume $\tau\neq0$ and $\Delta=0$.  If the repeated eigenvalue is
$\lambda=\tau/2\neq0$, similarity and multiplication of all orbit samples by
$\lambda^{-1}$ reduce the construction to
\[
  C_0:=\begin{pmatrix}1&1\\0&1\end{pmatrix}.
\]
Define
\[
  K_1:=\begin{pmatrix}1&1\\0&1\end{pmatrix},\qquad
  K_2:=\begin{pmatrix}1&0\\1&1\end{pmatrix},\qquad
  K_3:=\begin{pmatrix}2&-1\\1&0\end{pmatrix}.
\]
The minor formed from the entries $(1,1),(1,2),(2,1)$ has determinant $2$,
so the three matrices are linearly independent.  With
\[
  U_1:=\begin{pmatrix}1&-1\\-1&2\end{pmatrix}
\]
one has $\det U_1=1$ and
\[
  U_1K_1=\begin{pmatrix}1&0\\-1&1\end{pmatrix},\qquad
  U_1K_2=\begin{pmatrix}0&-1\\1&2\end{pmatrix},\qquad
  U_1K_3=\begin{pmatrix}1&-1\\0&1\end{pmatrix}.
\]
These matrices are non-scalar with trace $2$ and determinant $1$, hence are
all in the nontrivial Jordan orbit of $C_0$.  Scaling back by $\lambda$
transports the same ambiguity to the original Jordan orbit.
\end{proof}

% =====================================================================
\section{Finite-field counting details}
\label{app:counting}
% =====================================================================

We give the finite-geometric counts underlying
Theorem~\ref{thm:exact-probabilities}.  Throughout, $q$ is odd.  We use
standard facts on nonsingular conics, nonsingular quadrics, and their
polarities over finite fields; see \cite{Hirschfeld1998}.

\subsection{Lines relative to a nonsingular conic}

We first record the elementary line counts that will be used twice.

\begin{lemma}[External, tangent, and secant lines]
\label{lem:conic-line-counts}
Let $N\subset\mathbb P^2(\mathbb F_q)$ be a nonsingular conic.  The
numbers of external, tangent, and secant lines to $N$ are,
respectively,
\[
  \frac{q(q-1)}2,
  \qquad
  q+1,
  \qquad
  \frac{q(q+1)}2.
  \tag{B.1}
\]
\end{lemma}

\begin{proof}
The conic has $q+1$ rational points, and each point has a unique tangent line,
so there are $q+1$ tangents.  Every unordered pair of distinct conic points
determines a unique secant, giving
\[
  \binom{q+1}{2}=\frac{q(q+1)}2
\]
secants.  Since $\mathbb P^2(\mathbb F_q)$ has $q^2+q+1$ lines, the remaining
$q(q-1)/2$ lines are external.
\end{proof}

\subsection{Regular semisimple four-sample counts}

Let $C$ be regular semisimple and define
\[
  C^\circ:=C-\frac{\tau}{2}I_2,
  \qquad
  \eta:=\det(C^\circ)=-\frac{\Delta}{4}\in\mathbb F_q^\times.
\]
Translation by the scalar part identifies the similarity orbit with
\[
  \{X\in\mathfrak{sl}_2(\mathbb F_q):\det X=\eta\}.
  \tag{B.2}
\]
Indeed, every matrix in (B.2) has the same separable characteristic polynomial
as $C^\circ$ and is therefore similar to it.  Homogenizing (B.2) gives the
nonsingular projective quadric
\[
  Q_\eta:
  \det X=\eta z^2
  \subset\mathbb P(\mathfrak{sl}_2\oplus\mathbb F_q z)
  \simeq\mathbb P^3.
\]
The plane at infinity $H_\infty=\{z=0\}$ meets $Q_\eta$ in the
nonsingular nilpotent conic
\[
  N_\infty:
  \det X=0
  \subset\mathbb P(\mathfrak{sl}_2).
\]
The closure is hyperbolic in the split case and elliptic in the nonsplit case.
Equivalently, adding the $q+1$ points of $N_\infty$ to the affine
orbit sizes in (8.1) gives $(q+1)^2$ and $q^2+1$, the numbers of rational
points on a hyperbolic and an elliptic quadric, respectively.

Every line $\ell\subset H_\infty$ is contained in $q+1$ projective planes, one
of which is $H_\infty$; hence exactly $q$ affine planes have line at infinity
$\ell$.  The following lemma determines which of those planes are tangent to
$Q_\eta$.

\begin{lemma}[Affine plane sections of the regular semisimple quadric]
\label{lem:regular-plane-sections}
Let $\ell\subset H_\infty$ be external, tangent, or secant to
$N_\infty$.

For the hyperbolic quadric, the numbers of tangent affine planes among the $q$
affine planes through $\ell$ are
\[
  0,
  \qquad 1,
  \qquad 2,
\]
respectively.  The unique tangent plane in the middle case is tangent at the
point of $N_\infty$, whereas the two tangent planes in the secant case
are tangent at affine points.

For the elliptic quadric, the corresponding numbers are
\[
  2,
  \qquad 1,
  \qquad 0.
\]
The two tangent planes over an external line are tangent at affine points, and
the unique tangent plane over a tangent line is tangent at the point at
infinity.
\end{lemma}

\begin{proof}
Write
\[
  X=\begin{pmatrix}a&b\\ c&-a\end{pmatrix}.
\]
After multiplying the defining equation by $-1$, the quadric and its conic at
infinity are
\[
  a^2+bc+\eta z^2=0,
  \qquad
  N_\infty:\ a^2+bc=0.
\]
The associated polar form is
\[
  B((a,b,c,z),(a',b',c',z'))
  =2aa'+bc'+cb'+2\eta zz'.
\]
The orthogonal group of the conic is transitive on lines of each of the three
types (see, for example, \cite{Hirschfeld1998}), so it is enough to consider
\[
  \ell_r:\ c+rb=0,\qquad z=0,
\]
where $r=0$, a nonzero square, or a nonsquare according as $\ell_r$ is
tangent, secant, or external.  Indeed, on $\ell_r$ the conic equation becomes
$a^2-rb^2=0$.

Let
\[
  v_r=(0,1,r,0),\qquad e_z=(0,0,0,1).
\]
Then $\ell_r=v_r^\perp\cap H_\infty$ and
\[
  \ell_r^\perp=\mathbb P\langle v_r,e_z\rangle.
\]
Under the polarity of the quadric, planes through $\ell_r$ correspond to
points of this polar line, and the plane is tangent precisely when its pole is
on the quadric.  Since $H_\infty$ is not tangent to the quadric, all such
points give affine tangent planes.  On $\ell_r^\perp$ the quadric equation is
\[
  r t^2+\eta s^2=0.
\]
For $r=0$ this has the unique projective solution $s=0$, namely $v_0\in
N_\infty$, giving the tangent-at-infinity case.

Suppose now that $r\ne0$.  There are two projective solutions exactly when
$-r/\eta$ is a square, and none otherwise.  The ambient quadric is hyperbolic
exactly in the split case, equivalently when $-\eta=\Delta/4$ is a square.
Thus in the hyperbolic case a secant line ($r$ square) gives two solutions and
an external line ($r$ nonsquare) gives none.  In the elliptic case the
conclusions are reversed.  When $r\ne0$, every solution has $s\ne0$, hence its
pole is affine.  This proves both the counts and the stated locations of
tangency.
\end{proof}

We can now derive the section tables rather than assume them.  A nontangent
plane meets the quadric in a nonsingular projective conic with $q+1$ rational
points.  If its line at infinity is external, tangent, or secant to
$N_\infty$, then the numbers of affine points are therefore
\[
  q+1,
  \qquad q,
  \qquad q-1.
\]
No three distinct points of a nonsingular conic are collinear, so a section
with $m$ affine points contains
\[
  t=m(m-1)(m-2)
\]
ordered noncollinear triples.

In the hyperbolic case, Lemmas~\ref{lem:conic-line-counts} and
\ref{lem:regular-plane-sections} give the following counts of nontangent
planes:
\[
\begin{array}{c|c|c|c}
\text{section type}&\#\text{planes}&m&t\\ \hline
\text{external conic}
&q^2(q-1)/2&q+1&(q+1)q(q-1)\\[1mm]
\text{infinity-tangent conic}
&q^2-1&q&q(q-1)(q-2)\\[1mm]
\text{secant conic}
&q(q+1)(q-2)/2&q-1&(q-1)(q-2)(q-3).
\end{array}
\tag{B.3}
\]
Indeed, an external line contributes all $q$ affine planes; a tangent line
contributes $q-1$ nontangent affine planes; and a secant line contributes
$q-2$.

The remaining hyperbolic sections are tangent.  There are $q(q+1)$ affine
points on the quadric, hence $q(q+1)$ planes tangent at affine points.  Such a
section is the union of two rational generator lines meeting at the affine
tangency point.  Each generator contributes $q$ affine points, so
\[
  m=2q-1.
\]
Subtracting the ordered triples lying entirely on one component from all
ordered triples of distinct affine points gives
\[
\begin{aligned}
 t
 &=(2q-1)(2q-2)(2q-3)-2q(q-1)(q-2)=6(q-1)^3.
\end{aligned}
\]
There are also $q+1$ tangent planes at the points of
$N_\infty$.  Their two generator lines meet only at the point at
infinity, so their affine parts are disjoint, giving $m=2q$ and
\[
\begin{aligned}
 t
 =(2q)(2q-1)(2q-2)-2q(q-1)(q-2)=6q^2(q-1).
\end{aligned}
\]
Thus the two tangent rows are
\[
\begin{array}{c|c|c|c}
\text{section type}&\#\text{planes}&m&t\\ \hline
\text{affine tangent pair}
&q(q+1)&2q-1&6(q-1)^3\\[1mm]
\text{infinity tangent pair}
&q+1&2q&6q^2(q-1).
\end{array}
\tag{B.4}
\]

For the elliptic quadric, Lemma~\ref{lem:regular-plane-sections} gives the
nontangent table
\[
\begin{array}{c|c|c|c}
\text{section type}&\#\text{planes}&m&t\\ \hline
\text{external}
&q(q-1)(q-2)/2&q+1&(q+1)q(q-1)\\[1mm]
\text{tangent at infinity}
&q^2-1&q&q(q-1)(q-2)\\[1mm]
\text{secant}
&q^2(q+1)/2&q-1&(q-1)(q-2)(q-3).
\end{array}
\tag{B.5}
\]
Here an external line contributes $q-2$ nontangent affine planes, a tangent
line contributes $q-1$, and a secant line contributes all $q$.  A tangent
plane to an elliptic quadric contains no noncollinear triple of rational
points: its projective section has only the point of tangency as a rational
point.  Hence tangent planes contribute zero to the four-sample count.

Four affine orbit points are affinely independent exactly when the first three
are noncollinear and the fourth lies outside their affine plane.  Therefore a
plane section containing $m$ affine points and $t$ ordered noncollinear triples
contributes
\[
  t\bigl(|\Orb(C)|-m\bigr)
  \tag{B.6}
\]
successful ordered quadruples.

In the split case, summing (B.6) over (B.3) and (B.4) gives
\[
  q^2(q^2-1)(q^4+3q^3-7q^2+11q-10)
\]
successful ordered quadruples.  Division by
$|\Orb(C)|^4=[q(q+1)]^4$ gives (8.2).  In the nonsplit case, summing over
(B.5) gives
\[
  q^2(q-2)(q-1)^2(q+1)(q^2-2q-1),
\]
and division by $|\Orb(C)|^4=[q(q-1)]^4$ gives (8.3).

\subsection{Trace-zero semisimple triples}

Assume now $\tau=0$ and $\nu\neq0$.  Write a traceless matrix as
\[
  X=\begin{pmatrix}a&b\\c&-a\end{pmatrix},
  \qquad
  Q(X):=\det X=-a^2-bc.
  \tag{B.7}
\]
The projective zero locus $Q=0$ is a nonsingular conic
$N\subset\mathbb P(\mathfrak{sl}_2)\simeq\mathbb P^2$.  For a
two-dimensional linear subspace $L\subset\mathfrak{sl}_2$, let $\ell=\mathbb
P(L)$ and
\[
  m_L:=|L\cap\Orb(C)|.
\]

\begin{lemma}[Orbit points on projective lines]
\label{lem:tracezero-line-sections}
The numbers of lines of each type are given by (B.1), and
\[
\begin{array}{c|c|c}
\text{line type}&m_L\text{ (split)}&m_L\text{ (nonsplit)}\\ \hline
\text{secant}&q-1&q-1\\[1mm]
\text{external}&q+1&q+1\\[1mm]
\text{tangent}&2q&0.
\end{array}
\]
\end{lemma}

\begin{proof}
A projective point $[X]\in\ell\setminus N$ contributes vectors to the
level set $Q(X)=\nu$ precisely when the value of $Q$ at that projective point has
the same square class as $\nu$; in that case there are exactly two such vectors,
which differ by sign.

If $\ell$ is secant, the restricted binary quadratic form is equivalent up to a
nonzero scalar to $xy$.  Away from the two zeros, the values are represented
by $\lambda t$ with $t\in\mathbb F_q^\times$, so exactly half of the $q-1$ remaining projective points have either
prescribed square class.  Hence
$m_L=q-1$ in both regimes.

If $\ell$ is external, the restricted form is anisotropic and is equivalent,
up to a nonzero scalar, to the norm form from $\mathbb F_{q^2}$ to
$\mathbb F_q$.  On
$\mathbb F_{q^2}^\times/\mathbb F_q^\times$, which has $q+1$ elements, the
square class of the norm is well defined because rescaling by
$\mathbb F_q^\times$ changes the norm by a square.  The induced map to
$\mathbb F_q^\times/(\mathbb F_q^\times)^2$ is surjective, so each square
class occurs on $(q+1)/2$ projective points.  Thus $m_L=q+1$ in either regime.

Finally consider a tangent line.  By conjugation we may take the tangency
point to be $[E_{12}]$.  The tangent line is then $c=0$, and (B.7) restricts
to
\[
  Q=-a^2.
\]
Thus the $q$ points of the tangent line other than the point of tangency
have determinant in the square class of $-1$.  Since the trace-zero characteristic polynomial has
discriminant $-4\nu$, the orbit is split precisely when $-\nu$ is a square, i.e.
precisely when $\nu$ has the square class of $-1$.  Therefore these $q$ projective points contribute two representatives
in the split case and none in the nonsplit case, giving $2q$ and $0$.
\end{proof}

For a fixed $L$, each projective orbit point contributes exactly the two
vectors $K$ and $-K$.  Hence, after choosing the first vector in
$L\cap\Orb(C)$, exactly two choices for the second vector are dependent
with it.  The number of ordered independent pairs is therefore
\[
  m_L(m_L-2).
\]
The third vector must lie outside $L$, leaving $|\Orb(C)|-m_L$ choices.
Each ordered spanning triple determines $L$ uniquely from its first two vectors,
so summing
\[
  m_L(m_L-2)\bigl(|\Orb(C)|-m_L\bigr)
\]
over all projective lines counts every successful ordered triple exactly once.
Using Lemmas~\ref{lem:conic-line-counts} and
\ref{lem:tracezero-line-sections}, the split case gives
\[
  q(q-1)(q+1)(q^3+2q^2-4q-1)
\]
successful ordered triples, and division by $[q(q+1)]^3$ gives (8.4).  The
nonsplit case gives
\[
  q(q-1)^2(q+1)(q^2-3q+1),
\]
and division by $[q(q-1)]^3$ gives (8.5).

\subsection{Nilpotent triples}

For nonzero nilpotent $C$, the orbit has $q^2-1$ elements and its
projectivization is exactly the nilpotent conic $N$.  Each of the
$q+1$ conic points has $q-1$ nonzero scalar representatives, all of which lie
in the same nonzero nilpotent orbit.  A projective line meets a nonsingular
conic in at most two points, so three orbit elements span
$\mathfrak{sl}_2(\mathbb F_q)$ if and only if their three projective points are
distinct.

The first sample has $q^2-1$ choices.  The second must avoid the $q-1$
representatives of the first projective point, leaving
\[
  q^2-q
\]
choices.  The third must avoid the representatives over both previous
projective points, leaving
\[
  (q^2-1)-2(q-1)=(q-1)^2
\]
choices.  Hence the number of successful ordered triples is
\[
  (q^2-1)(q^2-q)(q-1)^2,
\]
and division by $(q^2-1)^3$ gives
\[
  \frac{q(q-1)}{(q+1)^2},
\]
which is (8.6).

\section*{}
\subsection*{Use of generative AI}
During the preparation of this work, the author used OpenAI's ChatGPT to
assist with language refinement, organizational review, bibliographic
verification, and consistency checking.  After using this tool, the author
reviewed and edited the content
as needed and takes full responsibility for the content of the publication.

%\printbibliography

\end{document}